\documentclass[reqno,11pt,a4paper]{amsart}
\usepackage{amsmath,amssymb,latexsym}
\usepackage{amscd}
\usepackage{enumerate}
\usepackage[all]{xy}
\usepackage{fancyhdr}
\usepackage{color}
\usepackage{comment} 
\numberwithin{equation}{section}

\usepackage[english]{babel}

\newcommand{\Z}{\mathbb Z}
\newcommand{\Q}{\mathbb Q}

\newcommand{\WH}{\widehat}
\newcommand{\loc}{\mathrm{loc}}

\newcommand{\QQ}{\mathbb{Q}}

\newcommand{\FF}{\mathbb{F}}

\newcommand{\Qp}{\mathbb{Q}_p}
\newcommand{\Zp}{\mathbb{Z}_p}
\newcommand{\Zl}{\mathbb{Z}_l}

\newcommand{\zp}{\Zp}

\newcommand{\GL}{\mathrm{GL}}
\newcommand{\Hom}{\mathrm{Hom}}

\newcommand{\sel}{\mathrm{Sel}}
\newcommand{\gal}{\mathrm{Gal}}

\newcommand{\coker}{\mathrm{Coker}}
\newcommand{\Img}{\mathrm{Im}}

\newcommand{\cyc}{\mathrm{cyc}}

\newcommand{\fcyc}{F_{\cyc}}
\newcommand{\lag}{\Lambda(\Gamma)}

\newcommand{\ilim}{\varinjlim}

\newcommand{\cO}{\mathcal{O}}
\newcommand{\cL}{\mathcal{L}}

\newcommand{\sss}{S_p^{\mathrm ss}}
\newcommand{\ssF}{S_{p,F}^{\mathrm{ss}}}
\newcommand{\Tr}{\mathrm{Tr}}
\newcommand{\rank}{\mathrm{rank}}
\newcommand{\Ep}{E_{p^\infty}}
\newcommand{\HIw}{H^1_{\mathrm{Iw}}}
\newcommand{\Finf}{F_\infty}
\newcommand{\Gal}{\mathrm{Gal}}
\newcommand{\ur}{\mathrm{ur}}

\usepackage[OT2,OT1]{fontenc}
\newcommand\cyr{%
\renewcommand\rmdefault{wncyr}%
\renewcommand\sfdefault{wncyss}%
\renewcommand\encodingdefault{OT2}%
\normalfont\selectfont}

\DeclareTextFontCommand{\textcyr}{\cyr}
\newcommand{\Sha}{{\mbox{\textcyr{Sh}}}}

\newtheorem{theorem}{Theorem}[section]
\newtheorem{proposition}[theorem]{Proposition}
\newtheorem{lemma}[theorem]{Lemma}
\newtheorem{corollary}[theorem]{Corollary}

\newtheorem{remark}[theorem]{Remark}

\definecolor{Green}{rgb}{0.0, 0.5, 0.0}

\newcommand{\Keywords}[1]{\par\noindent
{\small{Keywords and phrases}: #1}}

\newcommand{\AMS}[1]{\par\noindent
{\small{AMS Subject Classification}: #1}}

\author[A. Lei]{Antonio Lei}
\address{(A. Lei) D\'epartement de Math\'ematiques et de Statistiques, Universit\'e Laval, Pavillon Alexandre-Vachon, 1045 Avenue de la M\'edecine, Qu\'ebec, QC, Canada G1V 0A6}
\email{antonio.lei@mat.ulaval.ca}

\author[R. Sujatha]{R. Sujatha}
\address{(R. Sujatha) Mathematics Department, 1984, Mathematics Road, University of British Columbia,  Vancouver, 
Canada V6T 1Z2}
\email{sujatha@math.ubc.ca}

\begin{document}

\title{On Selmer groups in the supersingular reduction case}

\begin{abstract}
Let $p$ be a fixed odd prime. Let $E$ be an elliptic curve defined over a number field $F$ with good supersingular reduction at all primes above $p$. We study both the classical and plus/minus Selmer groups over the cyclotomic $\Zp$-extension of $F$. In particular, we give sufficient conditions for these Selmer groups to not contain a non-trivial sub-module of finite index. Furthermore, when $p$ splits completely in $F$, we  calculate the Euler characteristics of the plus/minus Selmer groups over the compositum of all $\Zp$-extensions of $F$ when they are defined. 
\end{abstract}

\maketitle
\let\thefootnote\relax\footnotetext{
\AMS{11R23, 11G05, 11F11}
\Keywords{Elliptic curves, modular forms, supersingular primes, Selmer groups}
}

\section{Introduction}\label{intro}
The theory of Galois representations for the absolute Galois group of a number field $F$ plays a central role in  Arithmetic Geometry.
Two important examples that  have been studied in depth are the Galois representations associated to an elliptic curve $E$ defined over 
$F$ (see \cite{serre}) and those arising from elliptic modular forms that originated in the results due to Eichler, Shimura,
Deligne and Serre (see \cite{D}).  There is a well-defined notion of these Galois representations being ordinary at a prime number $p$ (see  \cite{G}), which for simplicity we assume to be odd in this article. Iwasawa theory then intervenes effectively  thereby enabling a deeper study that provides arithmetic information associated to the Galois representations. The important objects of study in this situation are the Selmer groups associated to the Galois representations when they have good ordinary reduction at an odd prime $p$.  A central conjecture due to Mazur {\cite{m}} in the case of $p$-ordinary elliptic curves, states that the  dual Selmer group of the elliptic curve over the cyclotomic  $\Zp$-extension  is finitely generated  and torsion over the corresponding Iwasawa algebra. A generalised form of this conjecture is believed to be true for the Selmer groups of the Galois representation associated to modular forms  as well, when the representation is ordinary at $p$ {(which is equivalent to the $p$-th Fourier coefficient of the modular form being a $p$-adic unit). Kato showed in \cite{kato} that this conjecture holds  when $F/\QQ$ is an abelian extension.} Extensions of  Mazur's conjecture on the dual Selmer group being torsion over the cyclotomic extension  
have been studied over other Galois extensions $F_{\infty}$  of  $F$ whose Galois group  $\Gal(F_{\infty}/F)$ is a $p$-adic Lie group.
The main tools in these cases are Iwasawa theoretic in nature. It is known  that Mazur's conjecture on the dual Selmer group being a torsion module over the corresponding Iwasawa algebra is equivalent to the defining sequence for the 
Selmer group being short exact {(see \cite[Proposition~2.3]{H-M})}.  This equivalence can also be generalised to the Iwasawa theoretic study over other $p$-adic Lie extensions. This in turn enables us to study the Galois cohomology of the Selmer group over the corresponding Iwasawa algebra (see  {\cite{css}}) as well as the Euler characteristic of the Selmer group, when it is defined (see \cite{C-S}). 

We now turn to the case when the elliptic curve has supersingular reduction at all the primes of $F$ above $p$ and consider the Selmer group over the cyclotomic $\Zp$-extension  $F_{\cyc}$  of $F$.  The dual Selmer group over the cyclotomic extension, which is still finitely generated over the corresponding Iwasawa algebra  is then far from being torsion, and is conjectured to have positive rank ({\cite[Conjecture~2.5]{C-S}}). 
Nevertheless, it is pertinent to ask whether the corresponding  global to local map in the defining sequence for the Selmer group is exact.
All these can be suitably extended to the Galois representation arising from elliptic modular forms, and these questions are investigated in this paper. In the last couple of decades, there are important subgroups of the Selmer groups, which we simply refer to as the  `signed Selmer groups'  that have been defined by  Kobayashi in \cite{Ko} and later studied by various authors notably by Kim, Kitajima, Otsuki et al under the assumption that $a_p(E)=0$ (see \cite{Kim, Ki,K-O,ip,L,LP,L-Z}). These Selmer groups exhibit properties remarkably similar to the Selmer groups in the ordinary case and results on  their Galois cohomology and  Euler characteristic  are known ({\cite[Theorem~1.2]{Ki}}).   When $a_p(E)\ne0$, Sprung has  defined and extensively studied similar Selmer groups  in \cite{Sp1,Sp2}. In the present paper, we concentrate on the  $a_p(E)=0$ case, whenever we study  the signed Selmer groups, for simplicity.

Our main results in this paper  study the Selmer groups (see Theorem~\ref{thm:notorsion}) and the signed Selmer groups in the supersingular case (see Theorem~\ref{thm:pmSel})  along the lines  undertaken in \cite{C-S} for the ordinary case. Surprisingly, this lens affords  simpler proofs of known results and makes the Galois cohomological study of the Selmer groups more explicit. Additionally, we  consider the $p$-adic Lie extension $F_{\infty}$  defined as the compositum of all $\Zp$-extensions of $F$ {under the assumption that $p$ splits completely in $F$}.  Our techniques make it possible to study the Galois cohomology of the  signed Selmer groups over this extension and compute their Euler characteristics (see Theorem~\ref{thm:ES} below) in the case of 
elliptic curves. Thus this paper is largely a generalisation of  the work in \cite{C-S}  from the case of ordinary reduction to that of supersingular reduction. Apart from studying the analogues to the representations associated to modular forms (see Theorem~\ref{thm:MF}), our methods provide different and simpler proofs in the supersingular case, as well as the computation of the Euler characteristic over a larger $p$-adic Lie extension in a unified framework.

The paper consists of five sections including this introductory section. In section 2, we introduce and study the Galois cohomology of the Selmer groups  of elliptic curves in the supersingular case.  In section 3,  we frame an analogous study for the representations arising from elliptic modular forms and prove results in that context.  In section 4, we introduce the signed Selmer groups in the case of  elliptic curves under the assumption that $p$ is unramified in $F$
and study their Galois cohomology. In all three instances, we study the Selmer groups over $\fcyc$ and  give sufficient conditions for the Selmer groups to not contain a non-trivial sub-module of finite index. In section 5, we assume that $p$ splits completely in $F$ and consider the larger $p$-adic Lie extension  $F_{\infty}$ obtained as the compositum of all $\Zp$-extensions of $F$. We compute the Euler characteristic of the signed Selmer group over this extension when they are defined. The reason that we restrict our study to the case of elliptic curves in sections 4 and 5  is that the theory of signed Selmer groups for modular forms in the non-ordinary case are not very well understood.  We hope that our approach might be helpful towards understanding how the results can be extended to  modular forms and hope to return to this subject in a subsequent paper.

Since the present paper was submitted, we learnt that Ahmed and Lim  \cite{AL} independently studied the Euler characteristics of mixed-signed Selmer groups of an elliptic curve. They concentrated on signed Selmer groups over the cyclotomic $\Zp$-extension of a number field. It would be interesting to investigate whether one may combine our work with theirs to study mixed-signed Selmer groups of elliptic curves  over the compositum of all $\Zp$-extensions of a number field.

\section{Galois cohomology of the Selmer group}\label{sec:selcyc}

In this section, we study the Selmer group of an elliptic curve and the Galois cohomology of the Selmer group. Let $E$ be an elliptic curve defined over a number field $F$. It  is assumed throughout that $E$ has supersingular reduction at all the primes of $F$ lying above the odd prime $p$.  Let $S$ be the finite set of primes of $F$ lying over  $p$ and the set of primes where $E$ has bad reduction. Let $F_S$ denote the maximal extension of $F$ unramified outside $S$.  For any finite extension $L$  of $F$, and a prime $v$ of $L$, the groups $J_v(E/L)$ are defined as follows:
\begin{equation}\label{jv}
J_v(E/L) = \bigoplus_{w\mid v} \,  H^1(L_w, E)(p)
\end{equation}
where the direct sum is taken over all primes $w$ of $L$ lying above the prime $v$ of $F$. For an infinite extension ${\mathcal L}$ of $F$, we define
\begin{equation}\label{jvl}
J_v(E/{\mathcal L}) = \ilim\, J_v(E/L),
\end{equation}
where the inductive limit is taken over finite extensions $L$ of $F$ inside ${\mathcal L}$.
The ($p$-adic) Selmer group of $E$ over a finite extension $L$ denoted $\sel_p(E/L)$ is defined 
as the kernel of the natural map
\begin{equation}\label{sel}
 \lambda_L : H^1(F_S/L,E_{p^\infty})  \longrightarrow \underset{v\in S}{\bigoplus} J_v(E/L).
\end{equation}
 For an infinite extension ${\mathcal L}$ of $F$, the Selmer group $\sel_p (E/{\mathcal L})$ is 
 defined by taking the direct limits of $\sel_p(E/L)$  as $L$ varies over finite extensions  $L$ in ${\mathcal L}.$
 If ${\mathcal L}$ is contained in $F_S$, there is  an exact sequence
 \begin{equation}\label{selL}
 0 \to \sel_p (E/{\mathcal L}) \to H^1(F_S/{\mathcal L}, E_{p^{\infty}})  \overset{\lambda_{\mathcal L}}{\longrightarrow}  \underset{v\in S}{\bigoplus}
 J_v(E/{\mathcal L}).
 \end{equation}

Let us first consider the case when ${\mathcal L}$ is the cyclotomic  $\zp$-extension $\fcyc$ of $F$, with $\Gamma :=\gal(\fcyc/F).$  Denote the map $\lambda_{\fcyc}$ occurring in  \eqref{selL} by $\lambda_{\cyc}$. Let $X(E/\fcyc)$ be the Pontryagin dual of $\sel_p(E/\fcyc)$. Then $X(E/\fcyc)$  is a finitely generated  module over the Iwasawa algebra $\lag:=\Zp[[\Gamma]]$. When $E$ has good ordinary reduction at all the primes of  $F$ above $p$, under the hypothesis that $\sel_p(E/F)$ is finite, the  $\Gamma$-Euler characteristic  of $\sel_p(E/\fcyc)$ has been studied in the context of its relation to  the exact formula for the $L$-value predicted by the Birch and Swinnerton-Dyer conjecture.   Further, $X(E/\fcyc)$ being $\lag$-torsion  is equivalent to the surjectivity of the map $\lambda_{\cyc}$ 
occurring in \eqref{selL}. Another important  result that plays a key role  in computing the Euler characteristic is the vanishing of  $H_1(\Gamma, X(E/F_{\cyc}))$  under certain additional hypothesis (see \cite[Theorem 3.11]{C-S}). 

The big difference in the case when $E$ has supersingular reduction at the primes of $F$ above $p$ is that $X(E/\fcyc)$ is  no longer torsion as a $\lag$-module, even though it is finitely generated. Nevertheless,  in this section, we prove 
that $H_1(\Gamma, X(E/\fcyc))$ vanishes under the assumption that the map $\lambda_{\cyc}$ is surjective and 
\begin{equation}\label{WLC}
H^2(F_S/F, E_{p^{\infty}})=0,
\end{equation}
which can be considered as the analogue of the   Leopoldt Conjecture on number fields for $E_{p^\infty}$ over $F$. Furthermore, this is equivalent to the finiteness of the fine Selmer group of $E$ over $F$ as defined in \cite{fine} (see \cite[Lemma~3.2]{H}; we thank Meng Fai Lim for reminding us of this fact).

\begin{remark}
The equation \eqref{WLC} holds when both $E(F)$ and $\Sha(E/F)(p)$ are finite, as shown in \cite[Proposition~1.9]{C-S}.
\end{remark}

Our proof surprisingly follows the arguments in \cite{C-S}  for the ordinary reduction case and is rendered simpler in the supersingular case. It involves a detailed analysis of the so called fundamental diagram \eqref{fund} below that is used in proving control theorems in Iwasawa theory. 

%\begin{comment}
\begin{equation}\label{fund}
\begin{CD}
0 @>>>  \sel_p(E/\fcyc)^{\Gamma} @>>> H^1(F_S/\fcyc, E_{p^{\infty}})^{\Gamma} @>{{\lambda}^{\Gamma}_{\cyc}}>> \underset{v\in S}{\bigoplus}  J_v(E/\fcyc)^{\Gamma}\\
& & @A{\alpha}AA                                    @A{\beta}AA                                                       @A{\gamma=\oplus \gamma_v}AA\\
0 @>>> \sel_p(E/F) @>>> H^1(F_S/F, E_{p^{\infty}}) @>{\lambda}>> \underset{v \in S}{\bigoplus} J_v(E/F) .
\end{CD}
\end{equation}
%\end{comment}

This diagram gives a  natural map 
\begin{equation}\label{delta}
\delta : \ker \gamma \to \coker\,\lambda 
\end{equation}
and analysing this map is important in proving the surjectivity of $\lambda.$ We begin by collecting a few lemmas.

\begin{lemma}\label{lem:beta} In the diagram \eqref{fund} above,  the map $\beta$ is surjective and has finite kernel.
\end{lemma}
\begin{proof} The assertion follows from an easy application of the Hochschild-Serre spectral sequence and   
noting that the $p$-cohomological dimension of $\Gamma$ is one. Indeed, the  kernel of $\beta$ is isomorphic 
to $H^1(\Gamma, E_{p^{\infty}}(\fcyc))$ and this group is finite since $E_{p^{\infty}}(\fcyc)$ is finite, by a theorem of Imai \cite{I}.
From the Hochschild-Serre spectral sequence, it follows that the cokernel of $\beta$ is isomorphic to $H^2(\Gamma, E_{p^{\infty}}(\fcyc))$  which is zero as $\Gamma$ has $p$-cohomological dimension one.
\end{proof}

\begin{remark}
When $p$ is unramified over $F$,  then \cite[Proposition~8.7]{Ko} shows that $E_{p^\infty}(F_\cyc)=0$. In particular, $\beta$ is in fact an isomorphism.
\end{remark}

\begin{lemma}\label{kerg}
 In the diagram \eqref{fund} above, the map $\gamma$ is surjective. Furthermore, $\underset{v|p}{\bigoplus}J_v(E/F)\subset \ker\gamma$ and the containment is of finite index.
\end{lemma}
\begin{proof} 
When $v\nmid p$, it is shown in \cite[Lemma~1.11]{C-S} that $J_v(E/F)=H^1(F_v,E_{p^\infty})$ is finite, so $\ker\gamma_v$ is finite. The surjectivity of $\gamma_v$ is once again a consequence of the Hochschild-Serre spectral sequence and the fact that the $p$-cohomological dimension of $\Gamma$ is one.

Let us now consider the case $v|p$. Let $w$ be any prime of $F_\cyc$ above $v$ and write $\hat E$ for the formal group attached to $E$ at this prime. Since $E$ is supersingular at $w$, we have $E_{p^\infty}(F_{\cyc,w})=\hat E_{p^\infty}(F_{\cyc,w})$. In particular, we have
\[
H^1(F_{\cyc,w},E)(p)\cong \frac{H^1(F_{\cyc,w},E_{p^\infty})}{E(F_{\cyc,w})\otimes\Qp/\Zp}=\frac{H^1(F_{\cyc,w},\hat E_{p^\infty})}{\hat E(F_{\cyc,w})\otimes\Qp/\Zp}.
\]
The last quotient is zero by \cite[Proposition~4.3]{C-G}. Alternatively, we may consider the Pontryagin dual of $H^1(F_{\cyc,w},E)(p)$, which is given by the universal norm $\displaystyle \varprojlim_{L} \hat E(L)$, where $L$ runs through finite extensions of $F_v$ contained inside $F_{\cyc,w}$. This inverse limit is zero by \cite[Theorem 2]{Sc}. In particular, $\gamma_v$ is surjective and $\ker\gamma_v= J_v(E/F)$.
\end{proof}

\begin{remark}\label{rk:J}
The proof of Lemma~\ref{kerg} tells us that 
\[
\bigoplus_{v\in S}J_v(E/\fcyc)=\bigoplus_{v\in S,v\nmid p}J_v(E/\fcyc)=\bigoplus_{w|v\in S,v\nmid p}H^1(F_{\cyc,w},E_{p^\infty}).
\]
\end{remark}

\begin{proposition}\label{prop:equiv}
Suppose that $\lambda_\cyc$ in \eqref{selL} (with $\mathcal{L}=\fcyc$) is surjective. If $H^2(F_S/F,E_{p^\infty})=0$, then the following are equivalent:
\begin{enumerate}[(i)]
\item The map $\lambda_\cyc^\Gamma$ in \eqref{fund} is surjective;
\item $H^1(\Gamma,\sel_p(E/F_\cyc))=0$;
\item The map $\delta$ in \eqref{delta} is surjective.
\end{enumerate}
\end{proposition}
\begin{proof}
The assumption $H^2(F_S/F,E_{p^\infty})=0$, together with the Hochschild-Serre spectral sequence imply that $$H^1(\Gamma,H^1(F_S/\fcyc, E_{p^{\infty}}))=0$$ (c.f. \cite[Equation (33)]{C-S}). Therefore, the equivalence of $(i)$ and $(ii)$ follows from the $\Gamma$-Galois cohomology exact sequence attached to the short exact sequence
\[
0\rightarrow \sel_p(E/\fcyc)\rightarrow H^1(F_S/\fcyc, E_{p^{\infty}})\rightarrow \bigoplus_{v\in S}J_v(E/\fcyc)\rightarrow 0.
\]
$(i)\Rightarrow(iii)$: Consider the following commutative diagram:
\begin{equation}\label{eq:diagramdelta}
\begin{array}{cccl}
H^1(F_S/\fcyc, E_{p^{\infty}})^{\Gamma}   & \stackrel{\lambda_\cyc^\Gamma}{\longrightarrow} &\bigoplus_{v\in S}J_v(E/\fcyc)^\Gamma\\
\beta\Big\uparrow&  &  \gamma\Big\uparrow\\
H^1(F_S/F,E_{p^\infty})&\stackrel{\lambda}{\longrightarrow} & \bigoplus_{v\in S} (J_v(E/F))&\longrightarrow \coker\lambda\longrightarrow0.
\end{array}
\end{equation}
Lemmas~\ref{lem:beta} and \ref{kerg} say that $\beta$ and $\gamma$ are surjective. In particular, $\lambda^\Gamma_\cyc\circ\beta=\gamma\circ\lambda$ is surjective under condition $(i)$. By a simple diagram chase, there is a short exact sequence
\begin{equation}\label{eq:exactdelta}
\ker\gamma\stackrel{\delta}{\longrightarrow}\coker\lambda\longrightarrow\coker(\gamma\circ\lambda)\longrightarrow\coker\gamma=0,
\end{equation}
where the vanishing of $\coker\gamma$ is given by Lemma~\ref{kerg}.
Hence, $\delta$ is surjective and $(iii)$ holds.\\\\
\noindent
$(iii)\Rightarrow(i)$ If $\delta$ is surjective, then the surjectivity of $\gamma$ implies that $\gamma\circ\lambda$ is surjective. Hence, $\lambda_\cyc^\Gamma$ is also surjective by the commutative diagram \eqref{eq:diagramdelta}.
\end{proof}

The following proposition is a partial converse to the above proposition.
\begin{proposition}\label{partcon}
Assume  $H^2(F_S/F, E_{p^{\infty}})=0$. With notation as before, suppose that  the map $\lambda_\cyc^\Gamma$ in \eqref{fund} is surjective.
Then the map $\lambda_\cyc$ is surjective and $H^1(\Gamma,\sel_p(E/\fcyc))=0$.
\end{proposition}
\begin{proof} Put
$\Img_{\cyc}=\Img(\lambda_{\cyc}).$ The sequence
\eqref{selL} gives the two short exact sequences 
\begin{equation}\label{ssel1}
\begin{CD}
0 &@>>> \sel_p(E/\fcyc) &@>>> H^1(F_S/\fcyc, E_{p^{\infty}}) &@>{\lambda_{\cyc}}>>
 \Img_{\cyc} &@>>> 0\cr
0 & @>>> \Img_{\cyc}   &@>>> \bigoplus_{v\in S} J_v(E/\fcyc) &@>>> \coker \lambda_{\cyc}  &@>>>0.
\end{CD}
\end{equation}

The long exact $\Gamma$-Galois cohomology sequence associated to the first sequence, along with the Hochschild-Serre spectral sequence  gives  $H^1(\Gamma, \Img_{\cyc})=0.$
Using this  in the long exact $\Gamma$-cohomology sequence of the second short exact sequence in \eqref{ssel1}, it follows that the natural map
\[
\left(\bigoplus_{v\in S}\, J_v(E/\fcyc)\right)^{\Gamma }\to (\coker (\lambda_{\cyc}))^{\Gamma}
\]
is surjective. Consider the natural commutative diagram below 
\begin{equation}\label{eq:triangle}
\begin{array}{rcc}
H^1(F_S/\fcyc, E_{p^{\infty}})^{\Gamma}   & \longrightarrow &(\Img_{\cyc})^{\Gamma}\\
&\searrow   &  \downarrow\\
& &\left( \bigoplus_{v\in S}J_v(E/\fcyc)\right)^{\Gamma}.
\end{array}
\end{equation}
From the two $\Gamma$-Galois cohomology sequences associated to the two short exact sequences in \eqref{ssel1}, the vertical arrow is injective and the cokernel of the horizontal arrow is $H^1(\Gamma, \sel_p(E/\fcyc))$. 
Further, under the hypothesis of the surjectivity of $\lambda^{\Gamma}_{\cyc}$,  we conclude that
the vertical arrow in \eqref{eq:triangle} is an isomorphism. In particular, the horizontal map is surjective whence clearly
\begin{equation}\label{coker0}
\coker(\lambda_{\cyc})^{\Gamma}=0 , \quad {\text {and}} \quad H^1(\Gamma,\sel_p(E/\fcyc))=0.
\end{equation}
In particular, $\coker(\lambda_{\cyc})=0$ by Nakayama's lemma, which proves the surjectivity of $\lambda_\cyc$.
\end{proof}

\begin{corollary}\label{cor:injectdeltadual}
Assume that $H^2(F_S/F,E_{p^\infty})=0$ and that  $\Sha(E/F)(p)$ is finite. Then the surjectivity of 
$\delta$ is equivalent to the injectivity of 
\begin{equation}
E(F)^*\rightarrow \bigoplus_{v|p}E(F_v)^*\oplus\bigoplus_{v\nmid p} E_{p^\infty}(F_{\cyc,v})_{\Gamma_v},\label{eq:localization}
\end{equation}
where $M^*$ denotes the $p$-adic completion $\varprojlim M/p^n$ for an abelian group $M$. 
\end{corollary}
\begin{proof}
The proof is along  similar lines of the arguments in the orrdinary case considered in \cite{C-S}.
It follows from the proof of Proposition~1.9 in \cite{C-S} that $\coker\lambda$ is isomorphic to the 
Pontryagin dual $\widehat{E(F)^*}$ of $E(F)^*$.  By Lemma~\ref{kerg}, we have
\begin{equation}
\ker\gamma=\bigoplus_{v|p}J_v(E/F)\oplus \bigoplus_ {v\nmid p}\ker\gamma_v.
\end{equation}
For $v|p$, local Tate duality implies that
\[
\widehat{J_v(E/F)}=E(F_v)^*.
\]
For $v\nmid p$, the Hochschild-Serre spectral sequence says that 
\[
\ker\gamma_v=H^1(\Gamma_v,E_{p^\infty}(F_{\cyc,v}))=E_{p^\infty}(F_{\cyc,v})_{\Gamma_v},
\]
where the second equality is a consequence of the fact that $\Gamma_v$ is pro-cyclic.
Thus the dual of $\delta$ may be written as 
\[
\widehat{\delta}\,:\, E(F)^* \rightarrow \bigoplus_{v\mid p} E(F_v)^*\oplus\bigoplus_{v\nmid p} E_{p^\infty}(F_{\cyc,v})_{\Gamma_v}.
\]
By duality, the surjectivity of $\delta$ is equivalent to the injectivity of $\widehat{\delta}$. The result now follows.
\end{proof}
\begin{comment}
\begin{remark} Note that the arguments above show that in general,
the surjectivity of $\delta$ is equivalent to the injectivity of the dual map
\begin{equation}\label{eq:localization}
\widehat{\delta}:E(F)^*\rightarrow \bigoplus_{v\mid p} E(F_v)^*\oplus\bigoplus_{v\nmid p} E_{p^\infty}(F_{\cyc,v})_{\Gamma_v}.
\end{equation}
\end{remark}
\end{comment}

\begin{corollary}
Suppose that $E$ satisfies the following hypotheses:
\begin{enumerate}[(i)]
\item  $\Sha(E/F)(p)$ is finite;
\item The Mordell-Weil rank of $E/F$ is $\le 1$.
\end{enumerate}
Then $\lambda_\cyc$ is surjective, $H^2(F_S/F,\Ep)=0$ and the equivalent conditions in Proposition~\ref{prop:equiv} hold.
\end{corollary}
\begin{proof}
By the same proof of Theorem~12 in \cite{C-M} (where the authors worked over $\Q$ instead of a general number field $F$), (i) and (ii) imply that $H^2(F_S/F,\Ep)=0$.

Since the Mordell-Weil rank of $E/F$ is assumed to be $\le 1$, the quotient $E(F)^*$ is either $\Ep(F)$  or   $\Ep(F)\oplus\Zp\cdot P$ for some non-torsion $P\in E(F)$. Consequently, the localization map $E(F)^*\rightarrow E(F_v)^*$ is an injection for any $v|p$ (see also \cite[Lemma~8]{C-M}). In particular, the map in \eqref{eq:localization} has to be injective.

We can now apply Corollary~\ref{cor:injectdeltadual}, which tells us that $\delta$ is surjective. By \eqref{eq:exactdelta}, $\gamma\circ\lambda$ is also surjective. In particular, $\lambda_\cyc^\Gamma$ in \eqref{fund} is surjective. Proposition~\ref{partcon} gives the surjectivity of $\lambda_\cyc$. Proposition~\ref{prop:equiv} now applies.
\end{proof}

We now give a sufficient condition for the surjectivity of $\lambda_\cyc$.
\begin{lemma}\label{lem:PT}
If the localization map
\[
\loc:\varprojlim_n H^1(F_S/F_n,T_p(E))\rightarrow \bigoplus_{v|p}\varprojlim_n  H^1(F_{n,v},T_p(E))
\]
is injective. Then $\lambda_\cyc$ is surjective.
\end{lemma}
\begin{proof}
The description of $J_v(E/\fcyc)$ in Remark~\ref{rk:J} gives the following Poitou-Tate exact sequence
\[
0\rightarrow  \sel_p(E/\fcyc)\rightarrow H^1(F_S/\fcyc, E_{p^{\infty}})\rightarrow \bigoplus_{v\in S} J_v(E/\fcyc) \rightarrow \widehat{\ker\loc},
\]
from which the result follows.
\end{proof}

\begin{remark}\label{rk:surj}
 When $F=\QQ$, the injectivity of the localization map in Lemma~\ref{lem:PT} has been established by Kobayashi \cite[Theorem~7.3]{Ko}.
\end{remark}

\begin{remark}
We shall see in \S\ref{S:pm} that it is possible to deduce the surjectivity of $\lambda_\cyc$ from the theory of plus/minus Selmer groups when the latter are defined. See in particular Lemma~\ref{lem:surjpm} below.
\end{remark}

Another sufficient condition for the surjectivity of $\lambda_\cyc$ is the finiteness of $\sel_p(E/F)$.
\begin{proposition}\label{prop:lambdasur}
Suppose that $\sel_p(E/F)$ is finite, then $\lambda_\cyc$ is surjective.
\end{proposition}
\begin{proof}
Recall from \cite[proof of Proposition~3.9]{C-S} that the finiteness of $\sel_p(E/F)$ implies the surjectivity of the map
\[
\lambda'_\cyc: H^1(F_S/F,E_{p^\infty})\rightarrow \bigoplus_{v\in S,v\nmid p}J_v(\fcyc).
\]
 But $\lambda'_\cyc$ coincides with $\lambda_\cyc$ by Remark~\ref{rk:J}. Hence the result.
\end{proof}

\begin{theorem}\label{thm:notorsion}
Let $E$ be an elliptic curve over $F$ such that the following hypotheses are satisfied:
\begin{enumerate}[(i)]
\item $E$ has supersingular reduction at all the primes $v$ of $F$ lying above $p$.
\item The map $\lambda_\cyc$  in \eqref{selL} (with $\mathcal{L}=\fcyc$) is surjective.
\item $H^2(F_S/F, E_{p^{\infty}})=0.$
\item The equivalent conditions in Proposition~\ref{prop:equiv} and Corollary~\ref{cor:injectdeltadual} hold.
\end{enumerate}
Then the $\Lambda(\Gamma)$-module $ X(E/\fcyc)$ admits no non-trivial finite sub-module.
\end{theorem}
\begin{proof}
 The hypotheses implies that
 $$H_1(\Gamma, X(E/\fcyc))=0$$
 by duality. The result now follows from \cite[Lemma~A.1.5]{C-S}.
\end{proof}

\begin{remark}
In \cite[Proposition~4.1.1]{Gr}, Greenberg has given sufficient conditions for  $X(E/\fcyc)$ to contain no non-trivial finite sub-module. It is thus possible to obtain an alternative proof to Theorem~\ref{thm:notorsion} under different hypotheses. Indeed, our hypothesis on  the surjectivity of $\lambda_\cyc$ implies  the condition labelled $\mathrm{CRK}(\mathcal{D,L})$ in loc. cit., which says that 
$$ 
\rank_\Lambda \widehat{\sel_p(E/\fcyc)}+\rank_\Lambda\bigoplus_{v\in S}  \widehat{J_v(E/\fcyc)}=\rank_\Lambda\widehat{H^1(F_S/\fcyc, E_{p^{\infty}})}.
$$
The hypothesis $\mathrm{RFX}(\mathcal{D})$, which asserts that $T_p(E)\otimes \Lambda$ is a reflexive $\Lambda$-module, holds unconditionally. Under  the weak Leopoldt conjecture, i.e. $H^2(F_S/\fcyc,\Ep)=0$,   Greenberg's hypothesis $\mathrm{LEO(\mathcal{D}})$, which says that the kernel of the localization map $$H^2(F_S/\fcyc,\Ep)\rightarrow \bigoplus_{v\in\Sigma,w|v}H^2(F_{\cyc,w},\Ep)$$ is $\Lambda$-cotorsion, holds trivially.  Therefore, \cite[Proposition~4.1.1]{Gr} applies if we further assume
\begin{itemize}
\item $\mathrm{LOC}_v^{(1)}(\mathcal{D})$ holds for some non-archimedian prime $v$, that is, $(T_p(E)\otimes\Lambda)^{G_{F_v}}=0$;
\item $\mathrm{LOC}_v^{(2)}(\mathcal{D})$ holds for all $v\in S$, that is, the $\Lambda$-module $T_p(E)\otimes\Lambda/(T_p(E)\otimes\Lambda)^{G_{F_v}}$ is reflexive;
\item One of the following  conditions holds:
\begin{enumerate}[(a)]
\item $E[p]$ has no quotient isomorphic to $\mu_p$ as a $G_F$-module;
\item  There exists $v\in S$ such that $\mathrm{LOC}_v^{(1)}(\mathcal{D})$ holds and $\widehat{J_v(E/\fcyc)}$ is  a reflexive $\Lambda$-module.
\end{enumerate}
\end{itemize}
\end{remark}

\section{Modular forms}

Let $f=\sum a_n q^n\in S_k(\Gamma_0(N),\epsilon)$ be a normalized new cuspidal modular eigenform of level $N$. 
%Let $K$  be the completion of $\QQ(a_n:n\ge1)$ at a fixed prime above $p$. Write $\cO$ for the ring of integers of $K$ and fix a uniformizer $\varpi$. 
Assume that $(p,N)=1$.
 Let $E $ be the subfield of the complex numbers obtained by adjoining the Fourier coefficients of $f$ to $\Q$. 
It is well known that $E$ is a number field, and we denote its ring of integers by ${\mathcal O}.$  For each prime $l$, let
$$
{\mathcal O}_l ={\mathcal O } \otimes \Zl , \quad E_l = E\otimes {\mathbb Q}_l
$$
so that  ${\mathcal O}_l$ is the product of the completions  ${\mathcal O}_{\lambda}$ of ${\mathcal O}$ at the primes $\lambda$ of ${\mathcal O}$ dividing $l$.
Similarly $E_l$ is the product of the corresponding completions $E_{\lambda}$. 
By the results of Eichler, Shimura, Deligne, there is a system of $l$-adic  Galois representation
\begin{equation}\label{rhof}
\rho_f : G_{\Q} \longrightarrow  \GL_2({\mathcal O}_l) \subset \GL_2(E_l) 
\end{equation}
of $G_\QQ$ attached to $f$ (see \cite{D}, also  \cite{ribet} and \cite[\S8.3]{kato}). For simplicity, we denote by $V_f$ the corresponding 2-dimensional representation over $E_{\lambda}$, where $\lambda$ is a fixed prime of $E$ above $p$, and we fix a two dimensional lattice $T_f$ of $V_f$ so that $T_f$ is a free  ${\mathcal O}_{\lambda}$-module of rank two.  The representation is crystalline, with Hodge-Tate weights 0 and  $1-k.$  We assume that  $V_f$ is non-ordinary, that is $a_p\in\lambda$.

Fix an integer 
$j \in [1, k-1]$ and write $T=T_f(j)$, $V=V_f(j)$.
The corresponding divisible module $V/T$ is denoted by $A$, and $A$ is isomorphic to two copies of 
$E_{\lambda}/{\mathcal O}_{\lambda}$ equipped with a $G_{\Q}$-action. There is a tautological short exact sequence
\begin{equation}\label{eq:TVA}
0\rightarrow T\stackrel{\iota}{\longrightarrow} V\stackrel{\pi}{\longrightarrow}A\rightarrow 0.
\end{equation}
Given a number field $L$ and $v$ a place of $L$, we recall that Bloch and Kato \cite{bk} defined the canonical subgroups
\[
H^1_f(L_v,V):=
\begin{cases}
\ker\left(H^1(L_v,V)\rightarrow H^1(L_v,V\otimes \mathbb{B}_{\rm cris})\right)&v| p,\\
\ker\left(H^1(L_v,V)\rightarrow H^1(L_v^{\mathrm ur}, V) \right)&v\nmid p,
\end{cases}
\]
where $L_v^{\mathrm ur}$ denotes the maximal unramified extension of $L_v$ and
\begin{align*}
H^1_f(L_v,T)&:=\iota^*\left(H^1_f(L_v,V)\right);\\
H^1_f(L_v,A)&:=\pi_*\left(H^1_f(L_v,V)\right).
\end{align*}

If $q$ is a rational prime, we write
\[
J_q(A/L)=\bigoplus_{v|q}\frac{H^1(L_v,A)}{H^1_f(L_v,A)},
\]
where the direct sum is taken over all primes of $L$ over $q$. If $\cL$ is an infinite extension of $\QQ$, we define 
\[
J_q(A/\cL)=\ilim J_q(A/L),
\]
where the inductive limit is taken over number fields $L$ that are contained inside $\cL$.

Let $S$ be the set of primes dividing $pN$.
 We have the Selmer group sitting inside the following exact sequence
 \begin{equation}\label{selLf}
 0 \to \sel (A/{\mathcal L}) \to H^1(\cL_S/{\cL},A)  \overset{\lambda_{\mathcal L}}{\longrightarrow}  \underset{v\in S}{\bigoplus}
 J_v(A/{\cL}).
 \end{equation}
 
 Fix a number field $F$ and let $F_\cyc$ be the cyclotomic $\Zp$-extension of $F$ as before.
Throughout, we assume that the analogue of \eqref{WLC} holds for  $A$ over $F$,  namely
\begin{equation}
H^2(F_S/F, A)=0.\label{WLCf}
\end{equation}
Similar to \eqref{fund}, we have the following fundamental diagram:
\begin{equation}\label{fundf}
\begin{CD}
0 @>>>  \sel(A/\fcyc)^{\Gamma} @>>> H^1(F_S/\fcyc, A)^{\Gamma} @>{{\lambda}^{\Gamma}_{\cyc}}>> \underset{v\in S}{\bigoplus}  J_v(A/\fcyc)^{\Gamma}\\
& & @A{\alpha}AA                                    @A{\beta}AA                                                       @A{\gamma=\oplus \gamma_v}AA\\
0 @>>> \sel(A/F) @>>> H^1(F_S/F,A) @>{\lambda}>> \underset{v \in S}{\bigoplus} J_v(A/F) .
\end{CD}
\end{equation}
As in \eqref{delta}, this diagram gives rise to a natural map
\begin{equation}\label{deltaf}
\delta : \ker \gamma \to \coker\,\lambda .
\end{equation}

\begin{lemma}
In the diagram~\eqref{fundf} above, the map $\beta$ is surjective and has finite kernel.
\end{lemma}
\begin{proof}
As in Lemma~\ref{lem:beta}, the surjectivity is a consequence of Hochschild-Serre. The finiteness of $\ker\beta$ is proven in  \cite[Lemma~2.2]{S-MF}.
\end{proof}

Below is a generalization of Lemma~\ref{kerg}.

\begin{lemma}\label{kergf}
 In the diagram \eqref{fundf} above, the map $\gamma$ is surjective. Furthermore, $J_p(A/F)\subset \ker\gamma$ and the containment is of finite index.
\end{lemma}
\begin{proof}
For $v\nmid p$, the finiteness of $\ker\gamma_v$ is studied in \cite[proof of Proposition~3.3]{H-L}.
 If $w|p$ is a place of $\fcyc$, Perrin-Riou and Berger  showed that  $\displaystyle\varprojlim_{L\subset F_{\cyc,w}} H^1_f(L,T)=0$ (see \cite{PR,berger}), which generalizes the vanishing of the universal norm for supersingular elliptic curves. Therefore, if we replace $\hat E(L)$ by $H^1_f(L,T)$, the proof of Lemma~\ref{kerg} generalizes to the setting of modular forms.
\end{proof}

The analogue of Proposition~\ref{prop:equiv} holds.
\begin{proposition}\label{prop:equif}
Suppose that $\lambda_\cyc$ in \eqref{selLf}  (with $\mathcal{L}=\fcyc$) is surjective. If $H^2(F_S/F,A)=0$, then the following are equivalent:
\begin{enumerate}[(i)]
\item $\lambda_\cyc^\Gamma$ in \eqref{fundf} is surjective;
\item $H^1(\Gamma,\sel(A/F_\cyc))=0$;
\item The map $\delta$ in \eqref{deltaf} is surjective.
\end{enumerate}
\end{proposition}

We prove a sufficient condition for the equivalent conditions in Proposition~\ref{prop:equif}. 
\begin{lemma}\label{delsurjf}
Let $T^*=\Hom(T,\cO_\lambda)$ denote the linear dual of $T$ and define
\[
\Sha_p(T^*(1)/F):=\ker \left(H^1(F,T^*(1))\rightarrow\prod_{v|p} H^1(F_v,T^*(1))\right).
\]
Suppose that the the following hypotheses hold for $f$ and $F$.
\begin{enumerate}[i)]
\item $\Sha_p(T^*(1)/F)=0$.
\item $H^2(F_S/F,A)=0$.
\end{enumerate}
Then the map $\delta$ in \eqref{deltaf}  is surjective.
\end{lemma}
\begin{proof}
On the one hand, the Poitou-Tate exact sequence and \eqref{WLCf} imply that $$\widehat{\coker \lambda}\cong H^1_f(F,T^*(1)).$$ On the other hand, \cite[Proposition~3.8]{bk} together with local Tate duality give
\[
\widehat{\frac{H^1(F_v,A)}{H^1_f(F_v,A)}}\cong H^1_f(F_v,T^*(1))
\]
for all places $v|p$.
Hypothesis (i) implies that the localization map
\[
H^1_f(F,T^*(1))\rightarrow \prod_{v|p}H^1_f(F_v,T^*(1))
\]
is injective. In particular, Lemma~\ref{kergf} implies that
\[
\widehat{\delta}:\widehat{\coker\lambda}\rightarrow\widehat{\ker\gamma}= \widehat{J_p(A/F)}\oplus\bigoplus_{v\in S\setminus\{p\}}\widehat{\ker\gamma_v} 
\]
is injective. Hence, $\delta$ is surjective.
\end{proof}
Finally, we conclude this section with the following analogue of Theorem~\ref{thm:notorsion}.
\begin{theorem}\label{thm:MF}
Let $f$ be a normalized new cuspidal modular eigenform of level $N$ with $p\nmid N$. Suppose that the following hypotheses are satisifed:
\begin{enumerate}[i)]
\item $f$ is non-ordinary at $p$.
\item The map $\lambda_\cyc$ in \eqref{selLf} (with $\mathcal{L}=\fcyc$) is surjective.
\item $H^2(F_S/F, A)=0.$
\item The equivalent conditions in Proposition~\ref{prop:equif} hold.
\end{enumerate}
Then the  $\Lambda(\Gamma)$-module $ \widehat{\sel(A/\fcyc)}$ admits no non-trivial finite sub-module.
\end{theorem}

\section{Signed Selmer groups}\label{S:pm}
Let $F$ be a number field and $F'$ a subfield of $F$. Let  $E/F'$ be an elliptic curve with good reduction at all primes above $p$. We write $\sss$ be the set of primes of $F'$ lying above $p$ where $E$ has supersingular reduction. Assume that $\sss\ne\emptyset$. Furthermore, for all  $v\in\sss$, assume  that 
\begin{enumerate}[(i)]
\item $F_v'=\Qp$;
\item $a_v=1+p-\# \tilde E(F'_v)=0$;
\item $v$ is unramified in $F$.
\end{enumerate}

The condition (i) is necessary to apply Honda's theory of formal group to define the plus and minus conditions on $E$ (c.f. \cite[\S3.1]{K-O}).

Recall that $\fcyc$ denotes the cyclotomic $\Zp$-extension of $F$. For each integer $n\ge0$, let $F_n$ denote the sub-extension of $\fcyc/F$ such that $F_n/F$ is a cyclic extension of degree $p^n$. Let $\ssF$ denote the set of primes of $F$ lying above those inside $\sss$. For each $v\in\ssF$,  define
\begin{align}
E^+(F_{n,v})&=\{P\in \hat E(F_{n,v})\big|\Tr_{n/m+1}P\in \hat E(F_{m,v})\, \forall \text{ even }m,0\le m\le n-1\},\label{def+}\\
E^-(F_{n,v})&=\{P\in \hat E(F_{n,v})\big|\Tr_{n/m+1}P\in \hat E(F_{m,v})\, \forall \text{ odd }m,0\le m\le n-1\},\label{defn-}
\end{align}
where $\Tr_{n/m+1}$ is the trace map from $\hat E(F_{n,v})$ to $\hat E(F_{m+1,v})$. By an abuse of notation, we have denoted the unique prime of $F_n$ lying above $v$ by $v$ as well. The plus and minus Selmer groups over $F_n$ are defined by
\[
\sel_p^\pm(E/F_n)=\ker\left(\sel_p(E/F_n)\rightarrow\bigoplus_{v\in\ssF}\frac{H^1(F_{n,v},E_{p^\infty})}{E^\pm(F_{n,v})\otimes \Qp/\Zp}\right),
\]
where $E^\pm(F_{n,v})\otimes \Qp/\Zp$ are identified with their images inside $H^1(F_{n,v},E_{p^\infty})$ via the Kummer map. By \cite[Proposition~3.32]{K-O}, there is an isomorphism of $\Lambda(\Gamma)$-modules
\begin{equation}\label{eq:quotientpm}
\varinjlim_n\frac{H^1(F_{n,v},E_{p^\infty})}{E^\pm(F_{n,v})\otimes \Qp/\Zp}\cong
\WH{\Lambda(\Gamma)}^{\oplus[F_v:\Qp]} 
\end{equation}
for all $v\in \ssF$.

\begin{remark}\label{rk:SameSel}
When $n=0$, the $\pm$-Selmer groups $\sel_p^\pm(E/F)$ coincide with $\sel_p(E/F)$ as $E^\pm(F_v)=\hat E(F_v)$.
\end{remark}
The plus and minus Selmer groups over $\fcyc$ are defined by
\[
\sel_p^\pm(E/\fcyc)=\varinjlim_n \sel_p^\pm(E/F_n).
\]
The Pontryagin dual of  $\sel_p^\pm(E/\fcyc)$ will be denoted by $X^\pm(E/\fcyc)$.

For any $n\ge 0$ and a prime $v\in\ssF$, define
\[
J_v^\pm(E/F_n) = \bigoplus_{w\mid v}\frac{H^1(F_{n,w},E_{p^\infty})}{E^\pm(F_{n,w})\otimes \Qp/\Zp}
\]
where the direct sum is taken over all primes $w$ of $F_n$ lying above the prime $v$ of $F$ and
\[
J_v^\pm(E/\fcyc) = \ilim\, J_v^\pm(E/F_n).
\]
If $v\notin\ssF$, define $J_v^\pm(E/F_n)=J_v(E/F_n)$ and $J_v^\pm(E/\fcyc)=J_v(E/\fcyc)$ as in \eqref{jv} and \eqref{jvl}.
There is a natural morphism
\begin{equation}\label{eq:lambdapm}
\lambda_\cyc^\pm:H^1(F_S/\fcyc,E_{p^\infty})\rightarrow \bigoplus_{v\in S}J_v^\pm(E/F_\cyc),
\end{equation}
and a fundamental diagram analogous to \eqref{fund}:
\begin{equation}\label{fundpm}
\begin{CD}
0 @>>>  \sel_p^\pm(E/\fcyc)^{\Gamma} @>>> H^1(F_S/\fcyc, E_{p^{\infty}})^{\Gamma} @>{{\lambda}^{\pm,\Gamma}_{\cyc}}>> \underset{v\in S}{\bigoplus}  J_v^\pm(E/\fcyc)^{\Gamma}\\
& & @A{\alpha^{\pm}}AA                                    @A{\beta}AA                                                       @A{\gamma^\pm=\oplus \gamma_v^\pm}AA\\
0 @>>> \sel_p(E/F) @>>> H^1(F_S/F, E_{p^{\infty}}) @>{\lambda}>> \underset{v \in S}{\bigoplus} J_v(E/F) .
\end{CD}
\end{equation}
This diagram gives a  natural map 
\begin{equation}\label{deltapm}
\delta^\pm : \ker \gamma^\pm \to \coker\,\lambda .
\end{equation}
Proposition~\ref{prop:equiv} generalizes verbatim to this setting:
\begin{proposition}\label{prop:equivpm}
Suppose that $\lambda_\cyc^\pm$ in \eqref{eq:lambdapm} is surjective. If $H^2(F_S/F,E_{p^\infty})=0$, then the following are equivalent:
\begin{enumerate}[(i)]
\item The map $\lambda_\cyc^{\pm,\Gamma}$ in \eqref{fundpm} is surjective;
\item $H^1(\Gamma,\sel_p^\pm(E/F_\cyc))=0$;
\item The map $\delta^\pm$ in \eqref{deltapm} is surjective.
\end{enumerate}
\end{proposition}

\begin{remark}
As in Lemma~\ref{lem:PT}, it is possible to deduce the surjectivity of $\lambda_\cyc^\pm$ using Poitou-Tate exact sequences when the following map is injective:
\[
\varprojlim_n H^1(F_S/F_n,T_p(E))\rightarrow \bigoplus_{v|p,v\notin \ssF}\varprojlim_n \frac{ H^1(F_{n,v},T_p(E))}{H^1_f(F_{n,v},T_p(E))}\oplus \bigoplus_{v\in \ssF}\varprojlim_n \frac{ H^1(F_{n,v},T_p(E))}{H^1_\pm(F_{n,v},T_p(E))},
\]
where $H^1_\pm(F_{n,v},T_p(E))$ is the exact annihilator of $E^\pm(F_{n,v})\otimes\Qp/\Zp$ under local Tate duality.  When $F=\QQ$ (so that $\ssF=\{p\}$), this map is indeed injective thanks to the cotorsionness of $\sel_p^\pm(E/\QQ_\cyc)$ (c.f.  \cite[Lemma~6.2]{L-Z}). In fact, we shall see below that the surjectivity of $\lambda_\cyc^\pm$ is directly related to the cotorsionness of $\sel_p^\pm(E/\QQ_\cyc)$ in a much more general setting.
\end{remark}

\begin{proposition}\label{prop:surj-pm}
 The $\Lambda(\Gamma)$-module  $\widehat{\sel_p^\pm(E/\fcyc)}$ is torsion if and only if the map $\lambda_\cyc^\pm$ in \eqref{eq:lambdapm} is surjective and $H^2(F_S/\fcyc,\Ep)=0$. 
\end{proposition}

\begin{proof}
Let $T=T_p(E)$. For $i=1,2$, define
\[
H^i_{\mathrm{Iw}}(F_S,T)=\varprojlim H^i(F_S/F_n,T).
\]
For $v|p$, define
\[
\HIw(F_v,T)=\varprojlim H^1(F_{n,v},T)
\]
and
\begin{align*}
H^1_{\mathrm{Iw},f}(F_{v},T)&=\varprojlim H^1_f(F_{n,v},T)\text{ (for $v\notin \ssF$),}\\
H^1_{\mathrm{Iw},\pm}(F_{v},T)&=\varprojlim H^1_\pm(F_{n,v},T)\text{ (for $v\in \ssF$).}
\end{align*}
Consider the Poitou-Tate exact sequence
\begin{align}
\HIw(F_S,T)\stackrel{\loc^\pm}{\longrightarrow }\bigoplus_{v|p,v\notin \ssF} \frac{ \HIw(F_v,T)}{H^1_{\mathrm{Iw},f}(F_{v},T)}\oplus\bigoplus_{v\in \ssF}\frac{ \HIw(F_v,T)}{H^1_{\mathrm{Iw},\pm}(F_{v},T)}\notag\\
\rightarrow \widehat{\sel_p^\pm(E/\fcyc)}\rightarrow H^2_{\mathrm{Iw}}(F,T)\rightarrow 0\label{eq:PT}
\end{align}
(c.f. \cite[Proposition~A.3.2]{PRbook}).

As explained in \cite[page 23]{C-S},
\[
\rank_{\Lambda(\Gamma)}\frac{ \HIw(F_v,T)}{H^1_{\mathrm{Iw},f}(F_{v},T)}=[F_v:\Qp]
\]
for all $v|p$ and $v\notin \ssF$. For all $v\in \ssF$,
local Tate duality and \cite[Theorem~3.34]{K-O} imply that
\[
\rank_{\Lambda(\Gamma)}\frac{ \HIw(F_v,T)}{H^1_{\mathrm{Iw},\pm}(F_{v},T)}=[F_v:\Qp].
\]
Furthermore,
\[
\rank_{\Lambda(\Gamma)}\HIw(F_S,T)-\rank_{\Lambda(\Gamma)}H^2_{\mathrm{Iw}}(F_S,T)=[F:\QQ]
\]
by the  Euler characteristic formula (see proof of Proposition~1.3.2 in \cite{PRbook}). Therefore, on considering the rank of each term in \eqref{eq:PT}, we deduce that
\begin{equation}\label{eq:compareranks}
\rank_{\Lambda(\Gamma)}\ker\loc^\pm=\rank_{\Lambda(\Gamma)}\widehat{\sel_p^\pm(E/\fcyc)}.
\end{equation}

If $\lambda_\cyc^\pm$ is surjective and $H^2(F_S/\fcyc,\Ep)=0$,  the second Poitou-Tate exact sequence given in \cite[Proposition~A.3.2]{PRbook} implies that
\[
\widehat{ \coker\lambda_\cyc^\pm}\cong \ker\loc^\pm=0.
\]
Therefore, \eqref{eq:compareranks} tells us that  $\widehat{\sel_p^\pm(E/\fcyc)}$ is $\Lambda(\Gamma)$-torsion.

Conversely, suppose that  $\widehat{\sel_p^\pm(E/\fcyc)}$ is $\Lambda(\Gamma)$-torsion. Then \eqref{eq:PT} implies that $H^2_{\mathrm{Iw}}(F_S,T)$ is also $\Lambda(\Gamma)$-torsion. This is equivalent to $H^2(F_S/\fcyc,\Ep)=0$ (see \cite[Lemma~7.1]{lim} or \cite[Proposition~1.3.2]{PRbook}). This then gives the isomorphism
\[
\widehat{ \coker\lambda_\cyc^\pm}\cong \ker\loc^\pm.
\]via the Poitou-Tate exact sequence once again. In particular, \eqref{eq:compareranks} tells us that $\widehat{ \coker\lambda_\cyc^\pm}$ is $\Lambda(\Gamma)$-torsion. However, there is an injection 
\[
\widehat{\coker\lambda_\cyc^\pm }\hookrightarrow \bigoplus_{v|p} \widehat{J^\pm_v(E/\fcyc)}
\]
and that $\widehat{J^\pm_v(E/\fcyc)}$ is $\Lambda(\Gamma)$-torsion free for all $v|p$ (see \cite[page 38]{C-S} for $v\notin\ssF$ and \eqref{eq:quotientpm} for $v\in\ssF$). Therefore, $\widehat{\coker\lambda_\cyc^\pm }$ is  zero as required.
\end{proof}

\begin{remark}Suppose that $\sel_p(E/F)$ is finite. By Remark~\ref{rk:SameSel}, $\sel_p^\pm(E/F)$ are also finite. Then $\widehat{\sel_p^\pm(E/\fcyc)}$ are both  $\Lambda(\Gamma)$-torsion by the control theorem (c.f. \cite[Theorem~9.3]{Ko} and \cite[Lemma~3.9]{Ki}).
\end{remark}

\begin{lemma}\label{lem:surjpm}
If $\lambda_\cyc^\pm$ is surjective, then the map $\lambda_\cyc$ defined in \eqref{sel} is also surjective. Similarly, the surjectivity of $\lambda_\cyc^{\pm,\Gamma}$ in \eqref{fund} implies that of $\lambda_\cyc^{\Gamma}$ in \eqref{fundpm}.
\end{lemma}
\begin{proof}
By Remark~\ref{rk:J}, 
\[
\bigoplus_{v\in S}J_v^\pm(E/F_\cyc)=\bigoplus_{v\in S}J_v(E/F_\cyc)\oplus\bigoplus_{v\in \ssF}J_v^\pm(E/F_\cyc).
\]
Therefore, if $\lambda_\cyc^\pm$ is surjective onto $\bigoplus_{v\in S}J_v^\pm(E/F_\cyc)$, then $\lambda_\cyc$ is surjective onto $\bigoplus_{v\in S}J_v(E/F_\cyc)$. The proof for $\lambda_\cyc^\Gamma$ is the same.
\end{proof}

\begin{remark} In particular, under the assumption that $\lambda_\cyc^\pm$ is surjective and $ H^2(F_S/F,\Ep)=0$, there is an implication
\[
H^1(\Gamma,\sel_p^\pm(E/\fcyc))=0\quad \Rightarrow \quad H^1(\Gamma,\sel_p(E/\fcyc))=0.
\]
\end{remark}

We conclude this section with the following analogue of Theorem~\ref{thm:notorsion}.

\begin{theorem}\label{thm:pmSel}
Let $E$ be an elliptic curve over $F$ such that the following hypotheses are satisfied:
\begin{enumerate}[(i)]
\item $E$ has supersingular reduction at all the primes $v$ of $F$ lying above $p$.
\item The map $\lambda_\cyc^\pm$  is surjective.
\item $H^2(F_S/F, E_{p^{\infty}})=0$.
\item The equivalent conditions in Proposition~\ref{prop:equivpm} hold.
\end{enumerate}
Then the $\Lambda(\Gamma)$-module $\widehat{\sel_p^\pm(E/\fcyc)}$ admits no non-trivial finite sub-module.
\end{theorem}
\begin{proof}
 As in the proof of Theorem~\ref{thm:notorsion}, we have
 $$H_1(\Gamma, \widehat{\sel_p^\pm(E/\fcyc)})=0$$
 and the result  follows from \cite[Lemma~A.1.5]{C-S}.
\end{proof}

\section{Euler characteristics of plus and minus Selmer groups over abelian $p$-adic Lie extensions}

\subsection{Plus and minus norm groups over $\Zp^2$-extensions}
\label{S:pm2}
We review the plus and minus norm groups defined in \cite[\S2.1-\S2.2]{Kim}. Let $E/\Qp$ be an elliptic curve with good supersingular reduction and $\#\tilde E(\FF_p)=1+p$ and $k$ a finite unramified extension of $\Qp$. Let $k_\cyc$ denote the $\Zp$-cyclotomic extension of $k$. If $n\ge 0$ is an integer, let $k_n$ denote the unique sub-field of $k_\cyc$ such that $k_n/k$ is of degree $p^n$. The plus and minus norm groups $E^\pm(k_n)$ are defined as in \eqref{def+} and \eqref{defn-}.

Let $k^\ur$ denote the unramified $\Zp$-extension of $k$ and $k_\cyc^\ur$ the compositum of $k_\cyc$ and $k^\ur$. {In particular, $\Gal(k_\cyc^\ur/k)\cong \Zp^{2}$.} For an integer $m\ge0$, let $k^{(m)}$ denote the unique sub-extension 
of $k^\ur$, such that $k^{(m)}/k$ of degree $p^m$. 
{Similarly, for $n\ge1$, define $k_n^\ur$ to be the unramified $\Zp$-extension of $k_n$ and $k_n^{(m)}$ to the unique sub-extension such that $k_n^{(m)}/k_n$ of degree $p^m$. Note that $k_\cyc^\ur=\cup_{m,n} k_n^{(m)} $ and $\Gal(k_n^{(m)}/k)\cong \Z/p^m\times \Z/p^n$.}

 On replacing $k$ by $k^{(m)}$, \eqref{def+} and \eqref{defn-} give the plus and minus norm groups  $E^\pm\left(k_n^{(m)}\right)$ for all $n\ge0$. On taking union over $m$ and $n$, this gives
\[
E^\pm\left(k_\cyc^\ur\right)=\bigcup_{m,n\ge0}E^\pm\left(k_n^{(m)}\right).
\]

The following lemmas will be useful later.
\begin{lemma}\label{lem:notor}
$\Ep(k_\cyc^\ur)=0$.
\end{lemma}
\begin{proof}
This follows from \cite[Proposition~3.1]{K-O}, which says that $\Ep\left(k_\cyc^{(m)}\right)=0$ for all $m\ge0$. 
\end{proof}

\begin{lemma}\label{lem:cohopm}
Let $H_k=\Gal(k_\cyc^\ur/k_\cyc)${$\cong \Zp$}. Then,
\begin{itemize}
\item[(a)] $H^1(H_k,E^+(k_\cyc^\ur))\oplus H^1(H_k,E^-(k_\cyc^\ur))\cong H^1(H_k,\hat E(k^\ur)) $;
\item[(b)] $\left(E^\pm(k_\cyc^\ur)\otimes\Qp/\Zp\right)^{H_k}=E^\pm(k_\cyc)\otimes\Qp/\Zp$.
\end{itemize}
\end{lemma}
\begin{proof}
Recall from \cite[Proposition~2.6]{Kim} and \cite[Proposition~8.12]{Ko} the short exact sequence
\begin{equation}
\label{eq:pmexact}
0\rightarrow \hat{E}(k^\ur)\rightarrow E^+(k^\ur_\cyc)\oplus E^-(k^\ur_\cyc)\rightarrow\hat E(k^\ur_\cyc)\rightarrow 0,
\end{equation}
where the first map is given by the diagonal embedding $x\mapsto (x,x)$ and the second map is given by $(x,y)\mapsto x-y$. On taking $H_k$-cohomology, there is a long exact sequence
\begin{align*}
0\rightarrow \hat{E}(k^\ur)\rightarrow \hat E(k^\ur)\oplus\hat E(k^\ur)\rightarrow\hat E(k^\ur)\rightarrow H^1\left( H_k,\hat{E}(k^\ur)\right)\rightarrow\\
 H^1\left(H_k,E^+(k^\ur_\cyc)\right)\oplus H^1\left(H_k,E^-(k^\ur_\cyc)\right)\rightarrow H^1\left(H_k,\hat E(k^\ur_\cyc)\right).
\end{align*}
The first three terms clearly form a short exact sequence and \cite[Theorem~3.1]{C-G} says that $H^1\left(H_k,\hat E(k^\ur_\cyc)\right)=0$. Part (a) follows.

The $H_k$-cohomology exact sequence attached to the short exact sequence
\[
0\rightarrow \hat{E}(k^\ur)\rightarrow \hat{E}(k^\ur)\otimes\Qp\rightarrow \hat{E}(k^\ur)\otimes\Qp/\Zp\rightarrow 0
\]
gives the injectivity of  $H^1\left(H_k,\hat{E}(k^\ur)\right)\rightarrow H^1\left(H_k,\hat{E}(k^\ur)\otimes\Qp\right)$. On combining this with part (a), the natural maps  $H^1\left(H_k,E^\pm(k^\ur_\cyc)\right)\rightarrow H^1\left(H_k,E^\pm(k_\cyc^\ur)\otimes\Qp\right)$ are also injective. This implies part (b).
\end{proof}

\begin{proposition}\label{prop:free-quotient}
The Pontryagin dual of the quotient $\frac{H^1(k^\ur_\cyc,\Ep)}{E^{\pm}(k^\ur_\cyc)\otimes\Qp/\Zp}$ is a free $\Zp[[\Gal(k^\ur_\cyc/k)]]$-module of rank $[k:\Qp]$.
\end{proposition}
\begin{proof}
Let $m\ge0$ be an integer and write  $H_m=\Gal(k^\ur_\cyc/k^{(m)}_\cyc)${($=H_k^{p^m}\cong\Zp$)}. The proof for  \cite[(8.22)]{LP} gives
\[
\left(\frac{H^1(k^{\ur}_\cyc,\Ep)}{E^\pm(k^\ur_\cyc)\otimes\Qp/\Zp}\right)^{H_m}\cong\frac{H^1(k^\ur_\cyc,\Ep)^{H_m}}{\left(E^\pm(k^\ur_\cyc)\otimes\Qp/\Zp\right)^{H_m}}=\frac{H^1(k^{(m)}_\cyc,\Ep)}{E^{\pm}(k^{(m)}_\cyc)\otimes\Qp/\Zp},
\]
where the last equality follows from Lemma~\ref{lem:notor}, together with the inflation-restriction exact sequence and Lemma~\ref{lem:cohopm}(b).
By \cite[Proposition~3.32]{K-O}, this gives
\[
\left(\frac{H^1(k^{\ur}_\cyc,\Ep)}{E^\pm(k^\ur_\cyc)\otimes\Qp/\Zp}\right)^{H_m}\cong \WH{\Zp[[\Gal(k_\cyc/k)]]}^{\oplus[k^{(m)}:\Qp]}.
\]
The result follows.
 \end{proof}

\subsection{Plus and minus Selmer groups}\label{sec:Finf}
We work under the same setting as in \S\ref{S:pm} {and further assume that  $p$ splits completely in $F$}. Let $\Finf$ denote the compositum of all $\Zp$-extensions of $F$. Note that $\fcyc\subset \Finf$. The Galois group $\Omega:=\Gal(\Finf/F)$ is an abelian torsion-free pro-$p$ group. In particular, 
\[{\Omega\cong \Zp^d}\]
 for some integer $d\ge1$. {Leopoldt's conjecture predicts that $d$ is given by $r_2+1$, where $r_2$ is the number of pairs of complex embeddings of $F$.} We write $\Lambda(\Omega)$ for its Iwasawa algebra over $\Zp$, that is, $\Zp[[\Omega]]=\varprojlim \Zp[\Omega/U]$, where the inverse limit runs through all open subgroups $U$ of $\Omega$ and the connecting maps are projections.

Let $v$ be a place of $F$ and $w$ a place of $F_\infty$ above $v$.  If $v|p$, then $F_v$ may be identified with $\Qp$ and $F_{\infty,w}/F_v$ is an abelian pro-$p$ extension. By local class field theory, its Galois group is isomorphic to either $\Zp$ or $\Zp^2$, {depending on whether $d=1$ or $d>1$}. In particular,  $F_{\infty,w}$ coincides  with either $F_{v,\cyc}$ or $F_{v,\cyc}^\ur$ under the notation introduced in \S\ref{S:pm2}. If $v\nmid p$, then $F_{\infty,w}$ is the unique unramified $\Zp$-extension of $F_v$. 

Let $v_0\in\sss$ be a place of $F'$. Fix $v\in \ssF$ and a place $w$ of $\Finf$, both of which lying above $v_0$. {As noted above,  $F_{\infty,w}$ may be identified with either $F_{v,\cyc}$ or $F_{v,\cyc}^\ur$.} In both cases, there exist the plus and minus norm groups  $E^\pm(F_{\infty,w})\subset \hat{E}(F_{\infty,w})$. Define
\[
J_v^\pm (E/\Finf)=\bigoplus_{w|v} \frac{H^1(F_{\infty,w},\Ep)}{E^\pm(F_{\infty,w})\otimes\Qp/\Zp},
\]
where $E^\pm(F_{\infty,w})\otimes\Qp/\Zp$ is identified as a subgroup of $H^1(F_{\infty,w},\Ep)$ via the Kummer map as before.
The plus and minus Selmer groups over $\Finf$ are given by
\[
\sel^\pm_p(E/\Finf)=\ker\left(\sel_p(E/\Finf)\rightarrow\bigoplus_{v\in \ssF} J_v^\pm(E/\Finf) \right).
\]

\begin{remark}
In fact, it is  possible to define $2^{\#\ssF}$ signed Selmer groups since we may choose either the plus or minus norm subgroup as local condition for each prime $v\in\ssF$. In \cite{Kim}, Kim studied four such Selmer groups for an elliptic curve $E/\QQ$ when $F$ is an imaginary quadratic field where $p$ splits (so that $\#\ssF=2$). See also \cite{AL}.
\end{remark}

\subsection{Euler characteristics}

Recall that if $A$ is a discrete primary $\Gamma$-module, the $\Gamma$-Euler characteristic of $A$ is said to be finite if $H^0(\Gamma,A)$ and $H^1(\Gamma,A)$ are both finite and it is defined to be
\[
\chi(\Gamma,A)=\frac{\# H^0(\Gamma,A)}{\#H^1(\Gamma,A)}.
\]
\begin{theorem}[{\cite[Theorem~1.2]{Ki}}]
Suppose that $\sel_p(E/F)$ is finite and that $E$ has supersingular reduction at every prime above $p$, then the $\Gamma$-Euler characteristic of $\sel_p^\pm(E/\fcyc)$  is finite and equals, up to a $p$-adic unit
\[
\#\sel_p(E/F)\times\prod_v c_v,
\]
where $v$ runs through all primes of $F$ and $c_v$ denotes the Tamagawa number of $E$ at $v$.
\end{theorem}

We now study the $\Omega$-Euler characteristic of $\sel_p^\pm(E/\Finf)$, where $\Omega=\Gal(\Finf/F)$ and $\Finf$ is the extension considered in \S\ref{sec:Finf}. Let $H=\Gal(\Finf/\fcyc)
{\cong \Zp^{d-1}}$ (so that $\Gamma\cong \Omega/H$). {Let $w$ be a finite place of $F_\infty$. Recall that the decomposition group of $w|p$ in $H$ is either $\Zp$ or $0$, depending on whether $d>1$ or $d=1$. It is always $0$ when $w\nmid p$.}

Consider the following commutative diagram:

\begin{equation}\label{fundinf}
\begin{CD}
0 @>>>  \sel^\pm_p(E/\Finf)^H @>>> H^1(F_S/\Finf, E_{p^{\infty}})^H @>{{\lambda}^{H,\pm}_{\infty}}>> \underset{v\in S}{\bigoplus}  J_v^\pm(E/\Finf)^H\\
& & @A{\alpha_\infty^{\pm}}AA                                    @A{\beta_\infty}AA                                                       @A{\gamma_\infty^\pm=\oplus \gamma_{v,\infty}^\pm}AA\\
0 @>>> \sel^\pm_p(E/\fcyc) @>>> H^1(F_S/\fcyc, E_{p^{\infty}}) @>{\lambda_\cyc^{\pm}}>> \underset{v \in S}{\bigoplus} J_v^\pm(E/\fcyc) .
\end{CD}
\end{equation}

\begin{lemma}\label{lem:betainfty}
The map $\beta_\infty$ in \eqref{fundinf} is an isomorphism.
\end{lemma}
\begin{proof}
Lemma~\ref{lem:notor} says that $\Ep(F_{\infty,w})=0$ for any place $w$ of $F_\infty$ lying above a prime in $\ssF$. In particular, it implies that $\Ep(F_\infty)=0$. Therefore, $\beta_\infty$ is an isomorphism by the inflation-restriction exact sequence.
\end{proof}

\begin{corollary}
If $\lambda_{\cyc}^{\pm}$ is  surjective, then the map $\alpha_\infty^\pm$ in \eqref{fundinf} is injective and its cokernel is isomorphic to $\ker\gamma_\infty^\pm$.
\end{corollary}
\begin{proof}
This follows from the snake lemma and Lemma~\ref{lem:betainfty}.
\end{proof}

\begin{proposition}\label{prop:gamma-ss}
The map $\gamma_{v,\infty}^\pm$  in \eqref{fundinf} is an isomorphism for all $v\in\ssF$.
\end{proposition}
\begin{proof}
Let $w$ be a place of $\Finf$ lying above $v$ and write $H_w$ for the Galois group of the extension $F_{\infty,w}/F_{v,\cyc}$.  Then, the proof of Proposition~\ref{prop:free-quotient} gives
\[
\left(\frac{H^1(F_{\infty,w},\Ep)}{E^\pm(F_{\infty,w})\otimes\Qp/\Zp}\right)^{H_w}\cong\frac{H^1({F_{v,\cyc}},\Ep)}{E^\pm(F_{v,\cyc})\otimes\Qp/\Zp},
\]
which implies that $\gamma_{v,\infty}^\pm$ is the identity map.
\end{proof}

\begin{lemma}\label{lem:gamma-p-ord}
Let $v\in S\setminus\ssF$ such that $v|p$. Then $\gamma_{v,\infty}^\pm$ is surjective, 
and $\ker\gamma_{v,\infty}^\pm$ is  $\Lambda(\Gamma)$-cotorsion.
\end{lemma}
\begin{proof}Let $w$ be a place of $F_\infty$ lying above $v$.
Recall from \cite[Proposition~4.8]{C-G} that 
\[
\frac{H^1(F_{\infty,w},\Ep)}{E(F_{\infty,w})\otimes\Qp/\Zp}=H^1(F_{\infty,w},\tilde E_w[p^\infty]),\quad \frac{H^1(F_{v,\cyc},\Ep)}{E(F_{v,\cyc})\otimes\Qp/\Zp}=H^1(F_{v,\cyc},\tilde E_w[p^\infty]),
\]
where $\tilde E_w$ denotes the reduced curve of $E$ at $w$. By the inflation-restriction exact sequence, the cokernel of 
\[
H^1(F_{v,\cyc},\tilde E_w[p^\infty])\rightarrow H^1(F_{\infty,w},\tilde E_w[p^\infty])^{\Gal(F_{\infty,w}/F_{v,\cyc})}
\]
is given by $$H^2\left(F_{\infty,w}/F_{v,\cyc},\tilde E_w(\tilde F_{\infty,w})[p^\infty]\right),$$ where $\tilde F_{\infty,w}$ is the residue field of $ F_{\infty,w}$. This is zero since $F_{\infty,w}/F_{v,\cyc}$ is either trivial or a $\Zp$-extension.
The kernel of this map is given by $H^1(F_{\infty,w}/F_{v,\cyc},\tilde E_w(\tilde F_{\infty,w})[p^\infty])$, which is cofinitely generated over $\Zp$.  There are only finitely many  places of $\fcyc$ lying above $p$, so $\ker\gamma_{v,\infty}^\pm$ is also cofinitely generated over $\Zp$.
\end{proof}

\begin{lemma}\label{lem:gamma-not-p}
If $v\in S$ and $v\nmid p$, then the map $\gamma_{v,\infty}^\pm$ is an isomorphism.
\end{lemma}
\begin{proof}
Let $\nu$ be a place of $\fcyc$ lying above $v$ and let $w$ be a place of  $\Finf$ lying above $\nu$. Then, $F_{\infty,w}=F_{\cyc,\nu}$ is the unique $\Zp$-extension of $F_v$. In particular, $\gamma_{v,\infty}^\pm$ is nothing but the identity map.
\end{proof}

\begin{proposition}\label{prop:surjlambdaH}
Assume that  $\WH{\sel_p^\pm(E/\fcyc)}$ is $\Lambda(\Gamma)$-torsion. Then the map $\lambda_\infty^{H,\pm}$ in \eqref{fundinf} is surjective.
\end{proposition}
\begin{proof}
It suffices to show that $\gamma_\infty^\pm\circ\lambda_\cyc^\pm$ is surjective.
Proposition~\ref{prop:surj-pm} says that $\lambda_\cyc^\pm$ is surjective.
On combining Proposition~\ref{prop:gamma-ss} with Lemmas~\ref{lem:gamma-not-p} and \ref{lem:gamma-p-ord}, we deduce that $\gamma_\infty^\pm$ is surjective. Hence the result.
\end{proof}

\begin{corollary}\label{cor:H1Omega}
Assume that  $\sel_p(E/F)$ is finite, then 
\begin{enumerate}[(a)]
\item $H^1(H,\sel_p^\pm(E/\Finf))=0$;
\item $H^1(\Omega,\sel_p^\pm(E/\Finf))=H^1(\Gamma,\sel_p^\pm(E/\Finf)^H)$.
\end{enumerate}
\end{corollary}
\begin{proof}
Recall from Remark~\ref{rk:SameSel} that $\sel_p^\pm(E/F)=\sel_p(E/F)$. Thus, by Kim's control theorem (see Theorem~1.1 of \cite{Kim}; we remark that the hypothesis on $4\nmid[F_v:\Qp]$ in loc. cit. can be removed thanks to the work of Kitajima-Otsuki \cite{K-O}), $\sel_p^\pm(E/\Finf)^\Omega$ is finite.  By Nakayama's lemma (see \cite{BH}),   $\WH{\sel_p^\pm(E/F_\infty)}$ is a torsion $\Lambda(\Omega)$-module.
Thus, the same proof of Proposition~\ref{prop:surj-pm} shows that the map
\[
\lambda_\infty^\pm:H^1(F_S/\Finf,\Ep)\rightarrow \bigoplus_{v\in S}J_v^\pm(E/\Finf)
\]
is surjective and $H^2(F_S/\Finf,\Ep)=0$. 

Similarly, Kobayashi's control theorem \cite[Theorem~9.3]{Ko}  tells us that $\WH{\sel_p^\pm(E/\fcyc)}$ is $\Lambda(\Gamma)$-torsion. Thus, it follows from Proposition~\ref{prop:surjlambdaH} that $\lambda_\infty^{H,\pm}$ in \eqref{fundinf} is surjective. Therefore, the proof of Proposition~\ref{prop:equiv} gives
\[
H^1(H,\sel_p^\pm(E/\Finf))=0.
\]
The result now follows from the Hochschild-Serre spectral sequence
\[
H^m(\Gamma, H^n(H,\sel_p^\pm(E/\Finf))\Rightarrow H^{m+n}(\Omega,\sel_p^\pm(E/\Finf)).
\]
\end{proof}

\begin{corollary}
Assume that $E$ has supersingular reduction at every prime above $p$ and that $\sel_p(E/F)$ is finite. Then 
$$H^1(\Omega,\sel_p^\pm(E/\Finf))=0.$$
\end{corollary}
\begin{proof}
It follows from Lemma~\ref{lem:betainfty}, Proposition~\ref{prop:gamma-ss} and Lemma~\ref{lem:gamma-not-p} that the map $\alpha_\infty^\pm$ in \eqref{fundinf} is an isomorphism. Therefore, Corollary~\ref{cor:H1Omega}(b) says that
\[
H^1(\Omega,\sel_p^\pm(E/\Finf))=H^1(\Gamma,\sel_p^\pm(E/\fcyc)).
\]
This is zero by Proposition~\ref{prop:equivpm}.
\end{proof}

\begin{proposition}\label{prop:highervanishes}
Assume that $E$ has supersingular reduction at every prime above $p$ and that  ${\sel_p(E/F)}$ is finite. Then for $i\ge 2$, 
\[
H^i(\Omega,\sel_p^\pm(E/\Finf))=0.
\]
\end{proposition}
\begin{proof}
It is enough to show that
\begin{itemize}
\item[(a)] $H^i(\Omega,H^1_S(\Finf,\Ep))=0$ for $i\ge2$;
\item[(b)] $H^i(\Omega,J_v^\pm(E/\Finf))=0$ for $i\ge1$ and $v\in S$.
\end{itemize}
Since $p$ is odd, the cohomological dimension of $\Gal(F_S/F)$ is $2$. Therefore, (a) follows  from $H^2(F_S/\fcyc,\Ep)=0$ (which is a consequence of Proposition~\ref{prop:surj-pm}) and the spectral sequence
\[
H^i\left(\Omega,H^j(F_S/\Finf,\Ep)\right)\Rightarrow H^{i+j}(F_S/F,\Ep).
\]

For (b), consider the following two cases:
\begin{itemize}
\item[(b1)] $v\nmid p$;
\item[(b2)] $v\in \ssF$.
\end{itemize}
For (b1), recall from the proof of Lemma~\ref{lem:gamma-not-p} that the decomposition group of any place not dividing $p$ inside $H$ is trivial. Hence,
\[
H^1(\Omega, J_v^\pm(E/\Finf))=H^1(\Gamma,J_v^\pm(E/\fcyc))=\bigoplus_{w|v}H^1(\Gamma,H^1(F_{\cyc,w},\Ep))=0
\]
by the Hochschild-Serre spectral sequence. For (b2), the result follows from the co-freeness of $J_v^\pm(E/\Finf)$ as given by Proposition~\ref{prop:free-quotient}.
\end{proof}

\begin{theorem}\label{thm:ES}
Suppose that $E$ has supersingular reduction at every prime above $p$ and that $\sel_p(E/F)$ is finite. Then $\chi(\Omega, \sel_p^\pm(E/\Finf))$ is well-defined and equals $$\chi(\Gamma,\sel_p^\pm(E/\fcyc))=\#\sel_p(E/F)\times \prod_v c_v.$$
\end{theorem}
\begin{proof}
By Corollary~\ref{cor:H1Omega} and Proposition~\ref{prop:highervanishes}, $H^i(\Omega,\sel_p^\pm(E/\fcyc))=0$ for all $i\ge1$. The finiteness of $\sel_p(E/F)$, together with Kim's control theorem imply that $H^0(\Omega,\sel_p^\pm(E/\fcyc))$ is finite. In particular,
\[
\chi(\Omega,\sel_p^\pm(E/\Finf))=\#H^0(\Omega,\sel_p^\pm(E/\Finf)).
\]

Since all primes above $p$ are supersingular primes for $E$, Proposition~\ref{prop:gamma-ss}, Lemmas~\ref{lem:gamma-not-p} and \ref{lem:betainfty} imply that $\alpha_\infty^\pm$ in \eqref{fundinf} is an isomorphism. Therefore, 
\[
H^0(\Omega,\sel_p^\pm(E/\Finf))=H^0(\Gamma,\sel_p^\pm(E/\fcyc))
\]
as required.
\end{proof}

\section*{Acknowledgment}
We would like to thank Meng Fai Lim for helpful discussions and his comments  on an earlier version of the paper. We are also grateful to the anonymous referee for a  careful reading of our manuscript and suggestions, which led to many improvements to the paper.
This work was begun during the  PIMS Focus Period on Representations
in Arithmetic. The authors would like to thank PIMS for support and hospitality.
Both authors also gratefully acknowledge support of their respective
NSERC  Discovery Grants.

\end{document}